\documentclass[11pt]{article}
\usepackage[T1]{fontenc}
\usepackage{lmodern}
\usepackage[margin=1in]{geometry}
\usepackage{amsmath,amssymb,amsthm,mathtools}
\usepackage{enumitem,microtype,needspace,etoolbox}
\usepackage[colorlinks=true,linkcolor=blue!50!black,citecolor=blue!50!black,urlcolor=blue!50!black]{hyperref}
\usepackage{xcolor}
\hypersetup{pdftitle={On the square of an isosimple module},pdfsubject={On the square of an isosimple module},pdfkeywords={isosimple module, isoartinian module, quasi-projective module, endomorphism ring}}
\usepackage{fancyhdr}
\setlist[enumerate]{itemsep=2pt,topsep=4pt}
\newtheorem{theorem}{Theorem}[section]
\newtheorem{proposition}[theorem]{Proposition}
\newtheorem{lemma}[theorem]{Lemma}
\newtheorem{corollary}[theorem]{Corollary}
\theoremstyle{definition}

\newtheorem{example}[theorem]{Example}
\newtheorem{question}[theorem]{Question}
\theoremstyle{remark}
\newtheorem{remark}[theorem]{Remark}
\newcommand{\End}{\operatorname{End}}
\newcommand{\Hom}{\operatorname{Hom}}
\newcommand{\Ext}{\operatorname{Ext}}
\newcommand{\ann}{\operatorname{ann}}
\newcommand{\Tr}{\operatorname{Tr}}
\newcommand{\udim}{\operatorname{u.dim}}

\newcommand{\ess}{\leq_{\mathrm{ess}}}
\newcommand{\op}{\mathrm{op}}
\newcommand{\gl}{\operatorname{gl.dim}}
\newcommand{\Jac}{\operatorname{Jac}}
\title{On the square of an isosimple module}
\author{Ran Yan \thanks{School of Mathematics and Statistics, Tianshui Normal University, Tianshui 741001, China} \and Xiaolei Zhang  \thanks{Corresponding author} \thanks{School of Mathematics and Statistics, Tianshui Normal University, Tianshui 741001, China}}
\date{}
\BeforeBeginEnvironment{theorem}{\Needspace{9\baselineskip}}
\BeforeBeginEnvironment{proposition}{\Needspace{7\baselineskip}}
\BeforeBeginEnvironment{lemma}{\Needspace{7\baselineskip}}
\BeforeBeginEnvironment{corollary}{\Needspace{6\baselineskip}}
\begin{document}
	\maketitle
	\vspace{-1.5em}
	\begin{abstract}
		Let \(S\) be an isosimple module, that is, a nonzero module isomorphic to every nonzero submodule of itself. In this paper, we investigate  if the square \(S^2:=S\oplus S\) is isoartinian when \(S\) is an isosimple module. 
		If the base ring is commutative, then the answer is yes. However, for noncommutative rings this is no longer true  even when \(\End_R(S)\) is a discrete valuation ring.
		We also give some positive results coming from quasi-projectivity
		and Morita equivalence. Finally, we  consider when every
		finite powers of an isosimple module are isoartinian provided it holds for the square. Our related result  says that if \(S^2\)
		is isoartinian and \(E=\End_R(S)\) is either commutative Noetherian or right
		semihereditary, then every finite power of \(S\) is isoartinian.
	\end{abstract}
	
	
\section{Introduction and Preliminaries}
\subsection*{Introduction}

Recall that a nonzero module is called \emph{simple} if it has no nonzero proper
submodules. A module is called \emph{artinian} if every descending
chain of its submodules eventually becomes stationary. These are
basic notions in the structure theory of modules. Every simple
module is artinian, and every finite direct sum of simple modules
is artinian \cite{AF92,Lam99}. 
The notions of isosimplicity and isoartinianity arise by replacing
equality of submodules by isomorphism. A nonzero module \(S\) is
\emph{isosimple} if every nonzero submodule of \(S\) is isomorphic
to \(S\). A module \(M\) is \emph{isoartinian} if, for every
descending chain
\[
M_0\supseteq M_1\supseteq M_2\supseteq\cdots
\]
of submodules of \(M\), there is an integer \(d\geq 0\) such that
\(M_i\cong M_d\) for all \(i\geq d\). Thus the submodules need not
eventually be equal; only their isomorphism types must stabilize.
Every simple module is isosimple, every artinian module is
isoartinian, and every isosimple module is isoartinian. 

Facchini and Nazemian studied chain conditions up to isomorphism
in \cite{FN16,FN17} and developed further aspects of the theory
in \cite{FN19}. Their work connects these conditions with
generalizations of classical module structure, including the
isoradical and modules generated by isosimple modules.
Isosimple modules also appear under the name \emph{virtually
	simple modules} in \cite{BDV18a,BDV18b}. A module is called
\emph{virtually semisimple} if each of its submodules is
isomorphic to a direct summand. These related notions lead to
versions of familiar decomposition questions in which
isomorphism replaces equality. However, their direct-sum
properties do not follow formally from those of simple or
semisimple modules.

The fact that finite direct sums of simple modules are artinian
suggests asking whether direct sums of isosimple modules satisfy
the corresponding chain condition up to isomorphism. Facchini
and Nazemian formulated this question in
\cite[Open Problems~(2)]{FN19}:

\begin{quote}
	Is a direct sum of isosimple modules an isoartinian module?
\end{quote}

In this paper, we settle the question negatively. We give a
commutative example in which a direct sum of two nonisomorphic
isosimple modules is not isoartinian. We then prove that, over
a noncommutative ring, failure can occur even for two copies of
the same module. Thus the more restrictive question
\begin{equation}\label{eq:mainquestion}
	S\text{ isosimple}
	\quad\Longrightarrow\quad
	S^2:=S\oplus S\text{ isoartinian}?
\end{equation}
also has a negative answer in general. Our counterexample has
a faithful isosimple module over a complete local prime ring
that is Noetherian on both sides. Moreover, the endomorphism
ring of the module is a commutative discrete valuation ring.

The contents of the paper are as follows.   
Section~2 classifies isosimple modules over commutative rings:
they are precisely the modules \(R/\mathfrak p\) for which
\(R/\mathfrak p\) is a principal ideal domain. Consequently,
all their finite powers are isoartinian.   Section~3 gives a counterexample of (\ref{eq:mainquestion})  even when \(\End_R(S)\) is a discrete valuation ring. Section~4 proves that quasi-projective isosimple modules
have isoartinian finite powers. Morita equivalence preserves
the relevant properties and gives examples over matrix
rings over principal right ideal domains. We also show that
quasi-projectivity is not necessary.   
Finally, Section~5 considers whether
\begin{equation}\label{eq:powerquestion}
	S^2\text{ isoartinian}
	\quad\Longrightarrow\quad
	S^n\text{ isoartinian for every positive integer }n.
\end{equation}
The square hypothesis implies that \({}_ES\) is flat and
that every \(N\leq S^n\) has the form \(M\otimes_ES\), with
\(M\leq E^n\) finitely generated. We prove the implication
when \(E\) is commutative Noetherian or right semihereditary. 

All rings have identity, and all modules are unital. Unless otherwise stated,
modules are right modules. Let $R$ be a ring and  \(S\) be always an isosimple $R$-module. We write \(S^n\) for a finite direct sum, and 
\(E=\End_R(S)\).
Thus \(S\) is a left \(E\)-module, and \(\Hom_R(S,N)\) is a right \(E\)-module. By a principal right ideal domain we mean a domain in which every right ideal is principal.

	\subsection*{Preliminaries}
	
	We include the elementary properties on isosimple modules and isoartinian modules. Uniform
	dimension, essential submodules, and the Hopfian property are used in
	the usual sense \cite{AF92,Lam99}. The role of finite uniform dimension in
	quotient-ring theory goes back to Goldie \cite{Goldie60}.
	
	\begin{lemma}\label{lem:basic}
		An isosimple module \(S\) is cyclic, Noetherian and uniform. Every nonzero
		endomorphism of \(S\) is injective. Its endomorphism ring is a right Ore domain,
		and
		\[
		fE\subseteq gE\quad\Longleftrightarrow\quad fS\subseteq gS.
		\]
		The intersection of two nonzero principal right ideals of \(E\) is principal.
	\end{lemma}
	\begin{proof}
		Pick \(0\neq s\in S\). Since \(sR\cong S\), the module \(S\) is cyclic. Hence
		every submodule is either zero or cyclic, so \(S\) is Noetherian. If \(f\neq0\),
		then an isomorphism \(fS\to S\) composed with \(f\) gives a surjective
		endomorphism of the Noetherian module \(S\). It must be injective, so \(f\) is
		injective.
		
		If nonzero submodules \(U,V\) had zero intersection, choose \(f\) with image
		\(U\). Then \(V,fV,f^2V,\ldots\) form a direct sum: \(V\cap fS=0\), and
		injectivity of \(f\) gives the claim by induction. Their finite partial sums
		contradict the ascending chain condition. Therefore \(S\) is uniform.
		
		The factorization statement follows by using the inverse of the isomorphism
		\(g:S\to gS\). If \(f,g\neq0\), write \(fS\cap gS=hS\) with \(h\neq0\).
		Factoring through \(f\) and \(g\) gives \(h=fa=gb\), and the same argument for
		any element of \(fE\cap gE\) shows that this intersection is \(hE\). This also
		proves the right Ore condition.
	\end{proof}
	
	\begin{corollary}\label{cor:UFD}
		If \(\End_R(S)\) is commutative, then it is a UFD.
	\end{corollary}
	\begin{proof}
		Lemma~\ref{lem:basic} gives least common multiples as intersections of principal
		ideals. In a commutative domain this is the same as the GCD property. The
		correspondence \(fE\leftrightarrow fS\) gives the ascending chain condition on
		principal ideals. A GCD domain with this condition is a UFD; this is the usual
		factorization criterion \cite{Bourbaki}.
	\end{proof}
	
	\begin{lemma}\label{lem:projection}
		Every nonzero submodule \(N\leq S^n\) embeds, under some coordinate projection,
		as an essential submodule of \(S^r\), where \(r=\udim N\leq n\).
	\end{lemma}
	\begin{proof}
		Choose a set of coordinates of smallest possible size such that the projection
		onto them is injective on \(N\), and let \(L\leq S^r\) be its image. Removing
		any coordinate destroys injectivity. Therefore \(L\) has nonzero intersection
		with every coordinate axis. These intersections form an essential submodule of
		\(S^r\), because each coordinate is uniform. Hence \(L\ess S^r\), and its
		uniform dimension is \(r\).
	\end{proof}
	
	\begin{proposition}\label{prop:criterion}
		For every positive integer \(n\), the following are equivalent:
		\begin{enumerate}[label=\textup{(\roman*)}]
			\item \(S^n\) is isoartinian;
			\item every submodule of \(S^n\) is isomorphic to \(S^r\) for some
			\(0\leq r\leq n\).
		\end{enumerate}
		If these conditions hold, then every essential submodule of \(S^n\) is
		isomorphic to \(S^n\).
	\end{proposition}
	\begin{proof}
		Let \(L\ess S^n\). Its intersection with each axis contains a copy of \(S\), so
		there is a diagonal monomorphism \(F:S^n\to S^n\) with \(F(S^n)\subseteq L\).
		The chain
		\[
		S^n\supseteq L\supseteq F(S^n)\supseteq F(L)
		\supseteq F^2(S^n)\supseteq\cdots
		\]
		alternates between the isomorphism types of \(S^n\) and \(L\). Thus
		isoartinianity implies \(L\cong S^n\).
		
		Assume (i). Every smaller power \(S^r\) is isoartinian, since it is a submodule
		of \(S^n\). Apply Lemma~\ref{lem:projection} to a nonzero \(N\leq S^n\), and
		then apply the previous argument inside \(S^r\). This gives (ii).
		Conversely, under (ii) the integers \(\udim M_i\) in any descending chain
		eventually stabilize. The corresponding powers of \(S\) are then isomorphic,
		so the chain stabilizes up to isomorphism.
	\end{proof}
	
	\begin{remark}
		The criterion uses essential submodules and coordinate projections, not
		cancellation of arbitrary summands. In particular, a nonsplit short exact
		sequence cannot be treated as a direct-sum decomposition.
	\end{remark}
	
	\section{The commutative case}
	
In this section, we will show that the square \(S^2\) is always isoartinian when \(S\) is an isosimple module over a commutative ring.
	
	\begin{theorem}\label{thm:comm}
		Let \(R\) be commutative and \(S\neq0\) a right \(R\)-module. The following are
		equivalent:
		\begin{enumerate}[label=\textup{(\roman*)}]
			\item \(S\) is isosimple;
			\item there is a prime ideal \(\mathfrak p\) such that
			\(R/\mathfrak p\) is a principal ideal domain and
			\(S\cong R/\mathfrak p\).
		\end{enumerate}
		In this case every finite power \(S^n\) is isoartinian.
	\end{theorem}
	\begin{proof}
		Assume \(S\) is isosimple. By Lemma~\ref{lem:basic}, \(S\cong R/I\) for some
		ideal \(I\). Put \(D=R/I\). Endomorphisms of the cyclic module \(S\) are given
		by multiplication by elements of \(D\), so \(D\cong\End_R(S)\) is a domain.
		Every nonzero ideal \(J\) of \(D\) is isomorphic to \(D\). If \(\theta:D\to J\)
		is an isomorphism, then \(J=\theta(1)D\). Hence \(D\) is a principal ideal
		domain and \(I\) is prime.
		
		Conversely, every nonzero ideal of a principal ideal domain is isomorphic to the
		regular module, since it is generated by one element. Thus \(R/\mathfrak p\) is
		isosimple as an \(R\)-module. For every \(N\leq S^n\), the action factors
		through \(D\), and the submodule theorem for free modules over a PID gives
		\(N\cong D^r\) with \(r\leq n\). See \cite{Lam99} for this standard theorem.
		Now Proposition~\ref{prop:criterion} applies.
	\end{proof}
	
	This agrees with the commutative classifications of virtually simple and
	completely virtually semisimple modules in \cite{BDV18a,BDV18b}. The proof also
	shows why no Noetherian assumption on the original ring \(R\) is needed: it is
	the faithful quotient \(R/\ann_R(S)\), not \(R\) itself, that must be a PID.
	
	\begin{example}\label{ex:comm}
		For any commutative ring \(R\), every prime quotient that is a PID gives a
		positive example. This includes residue fields, polynomial quotients isomorphic
		to \(k[t]\), and quotients isomorphic to a discrete valuation ring. In
		particular, over \(k[x_1,\ldots,x_d]\), the module obtained from a surjection
		onto \(k[t]\) has isoartinian finite powers. Nilpotents or non-Noetherian
		behavior in the kernel do not affect the conclusion.
	\end{example}
	
	It is important that we only take copies of the same module.
	
	\begin{example}\label{ex:mixed}
		Let \(k\) be a field, \(R=k[x,y]/(xy)\), and
		\[
		U=R/(y)\cong k[x],\qquad V=R/(x)\cong k[y].
		\]
		Both modules are isosimple, but \(M=U\oplus V\) is not isoartinian.
		The map
		\[
		R\longrightarrow M,\qquad
		r\longmapsto(r+(y),r+(x))
		\]
		is injective. Its image is
		\[
		N=\{(f(x),g(y)):f(0)=g(0)\}.
		\]
		Thus \(N\) is cyclic. The module \(M\) is not cyclic, since
		\(M/(x,y)M\cong k^2\), whereas a cyclic \(R\)-module has a quotient
		of dimension at most one modulo \((x,y)\).
		
		Multiplication by \(a=x+y\) is injective on \(M\), and
		\(aM=xk[x]\oplus yk[y]\subseteq N\). The descending chain
		\[
		M\supseteq N\supseteq aM\supseteq aN
		\supseteq a^2M\supseteq a^2N\supseteq\cdots
		\]
		alternates between two different isomorphism types.
	\end{example}
	
	In particular, the positive result for \(S^2\) over commutative rings does not
	answer the question about all finite sums of isosimple modules. The decomposition
	results in \cite{BDV18b} concern modules that are already known to admit the
	relevant decompositions; they cannot be used to claim closure under arbitrary
	finite direct sums.
	
	\section{A noncommutative counterexample}
	
In this section, we will show that the square \(S^2\) may be not isoartinian even when \(S\) is an isosimple module over a non-commutative ring.	We use a family of orders in matrix rings. Their rings of fractions are simple
	artinian, but the residue algebra is deliberately smaller than a full matrix
	algebra. This discrepancy produces the obstruction. The quotient-ring viewpoint
	is standard in the theory of orders \cite{Goldie60,MR01}; all calculations
	needed below are explicit.
	
	Let \(L/k\) be a finite field extension of degree \(d>1\).
	Represent \(L\) in \(M_d(k)\) so that \(k^{1\times d}\) is its
	one-dimensional right regular module. Put
	\[
	O=k[[\pi]],\qquad F=k((\pi)),\qquad
	A=L+\pi M_d(O)\subseteq M_d(O),\qquad
	S=O^{1\times d}.
	\]
	The action on \(S\) is ordinary multiplication of row vectors by
	matrices.
	
	\begin{theorem}\label{thm:counter}
		The ring \(A\) is a noncommutative complete local prime ring, finite over the
		central DVR \(O\), and it is left and right Noetherian. The module \(S_A\) is
		faithful and isosimple, and
		\[
		\End_A(S)\cong O.
		\]
		Nevertheless \(S^2\) is not isoartinian.
	\end{theorem}
	\begin{proof}
		As an \(O\)-module, \(A\) lies between \(\pi M_d(O)\) and \(M_d(O)\), and it is
		the inverse image of \(L\) under reduction modulo \(\pi\). It is finite over
		\(O\), complete, and Noetherian on both sides. An element whose residue in
		\(L\) is nonzero has an inverse in \(M_d(O)\); its residue is the inverse in
		\(L\), so its inverse lies in \(A\). An element in \(\pi M_d(O)\) is not a
		unit. Therefore
		\[
		\Jac(A)=\pi M_d(O),\qquad A/\Jac(A)\cong L.
		\]
		Noncommutativity follows from the presence of \(\pi M_d(O)\).
		
		To prove primeness, if \(0\neq a,b\in A\), choose
		\(c\in M_d(F)\) with \(acb\neq0\). Multiplying \(c\) by a
		sufficiently high power of \(\pi\) puts it in \(\pi M_d(O)\subseteq A\).
		Thus \(aAb\neq0\). Faithfulness of the row representation is immediate.
		
		Let \(0\neq N\leq S_A\). The scalar matrices \(OI_d\) lie in \(A\), so \(N\)
		is an \(O\)-submodule. Choose the least valuation \(m\) of a nonzero coordinate
		of an element of \(N\). Then \(N\subseteq\pi^mS\), and the image of
		\(\pi^{-m}N\) in \(S/\pi S\) is nonzero. It is stable under \(L\). Since
		\(k^{1\times d}\) is simple as an \(L\)-module, this image is all of
		\(S/\pi S\). By Nakayama's lemma applied to the finite \(O\)-module
		\(S/\pi^{-m}N\), we get
		\[
		N=\pi^mS.
		\]
		Multiplication by \(\pi^m\) is an \(A\)-isomorphism \(S\to N\).
		
		An \(A\)-endomorphism of \(S\) is \(O\)-linear and hence is right
		multiplication by a matrix \(B\in M_d(O)\). It commutes with every
		\(\pi e_{ij}\). Since \(O\) is a domain, it commutes with all matrix units,
		so \(B\) is scalar. Conversely, scalar matrices give endomorphisms. This proves
		\(\End_A(S)=O\).
		
		Choose \(T\in L\setminus kI_d\), and define
		\[
		N_T=\{(u,v)\in S^2:\overline u=\overline vT\},
		\]
		where bars denote reduction modulo \(\pi\). This is an \(A\)-submodule,
		because \(L\) is commutative. Moreover
		\[
		\pi S^2\subsetneq N_T\subsetneq S^2.
		\]
		Any map \(S\to S^2\) has the form \(s\mapsto(as,bs)\), with \(a,b\in O\).
		Its image lies in \(N_T\) exactly when
		\[
		\overline a I_d=\overline bT.
		\]
		Since \(T\) is not scalar, this requires \(a,b\in\pi O\). Consequently the sum
		of all images of maps \(S\to N_T\) is
		\[
		\Tr(S,N_T)=\pi S^2\neq N_T.
		\]
		An isomorphism \(N_T\cong S^2\) would make \(N_T\) generated by \(S\), which
		is impossible. Finally
		\[
		S^2\supseteq N_T\supseteq\pi S^2\supseteq\pi N_T
		\supseteq\pi^2S^2\supseteq\cdots
		\]
		is a descending chain with alternating, distinct isomorphism types.
	\end{proof}
	
	\begin{example}\label{ex:explicit}
		Take \(k=\mathbb Q\), \(L=\mathbb Q(i)\), \(d=2\), and use
		\[
		L=\left\{
		\begin{pmatrix}a&b\\-b&a\end{pmatrix}:a,b\in\mathbb Q
		\right\}.
		\]
		Thus
		\[
		A=\left\{\begin{pmatrix}a&b\\-b&a\end{pmatrix}
		+\pi B:a,b\in\mathbb Q,\ B\in M_2(\mathbb Q[[\pi]])\right\}.
		\]
		With \(T=\left(\begin{smallmatrix}0&1\\-1&0\end{smallmatrix}\right)\),
		one may instead take the submodule defined by
		\(\overline v=\overline uT\). Sending a matrix of \(A\) to its two
		rows identifies this submodule with \(A_A\), so it is cyclic.
		On the other hand,
		\[
		S/S\Jac(A)\cong L,\qquad
		S^2/S^2\Jac(A)\cong L^2,
		\]
		and \(S^2\) cannot be cyclic. This gives a second direct proof that
		the two isomorphism types in the chain are different.
	\end{example}
	
	The example has several useful features. The module is cyclic and Noetherian,
	the ambient ring is prime and Noetherian, and the endomorphism ring is a
	commutative PID. None of these facts implies that the square is isoartinian.
	In particular, the condition that \(\End_R(S)\) be a principal right ideal
	domain should not be used as a substitute for quasi-projectivity in the next
	section.
	
	The construction also separates two questions about the endomorphism ring.
	Lemma~\ref{lem:basic} describes intersections of principal right ideals of
	\(E\). It does not say that every submodule of \(S^2\) is generated by maps
	from \(S\). The module \(N_T\) is an explicit failure of that generation
	property.
	
	\section{Affirmative cases over noncommutative rings}
	In this section, we will investigate when the square \(S^2\) is  isoartinian provided that \(S\) is an isosimple module over a non-commutative ring.
	\subsection{Quasi-projective modules and matrix rings}
	
	Recall that \(S\) is \emph{quasi-projective} if for every submodule \(U\),
	every map \(S\to S/U\) lifts through the natural map \(S\to S/U\). We use
	this definition rather than a projectivity assumption over the original ring.
	For the general theory see \cite{Wis91,Wis96}; the standard distinctions
	between projectivity and quasi-projectivity are also discussed in \cite{Lam99}.
	
	\begin{theorem}\label{thm:qp}
		If an isosimple module \(S\) is quasi-projective, then every submodule of
		\(S^n\) is isomorphic to a finite direct sum of copies of \(S\).
		In particular, all finite powers of \(S\) are isoartinian.
	\end{theorem}
	\begin{proof}
		Proceed by induction on \(n\). Let
		\(N\leq S\oplus S^{n-1}\), let \(U=N\cap(S\oplus0)\), and let
		\(V\) be the image in \(S^{n-1}\). By induction \(V\cong S^r\).
		If \(U=0\), then \(N\cong V\). Otherwise \(U\cong S\).
		
		There is a well-defined homomorphism
		\[
		h:V\longrightarrow S/U,\qquad
		h(v)=s+U\quad\text{if }(s,v)\in N.
		\]
		Write \(V=V_1\oplus\cdots\oplus V_r\), with each \(V_i\cong S\).
		Quasi-projectivity lifts each restriction \(h|_{V_i}\) to a map
		\(V_i\to S\). Adding these lifts gives \(\widetilde h:V\to S\).
		Then \(v\mapsto(\widetilde h(v),v)\) is a section of \(N\to V\),
		and \(N\cong U\oplus V\). The case \(U=S\) is included, with
		\(S/U=0\). Proposition~\ref{prop:criterion} finishes the proof.
	\end{proof}
	
	\begin{corollary}\label{cor:easypositive}
		The conclusion holds for simple modules and for projective isosimple
		modules. A quasi-injective isosimple module is simple, so it also has
		the conclusion.
	\end{corollary}
	\begin{proof}
		Simple modules are quasi-projective, since their only quotients are
		themselves and zero. Projective modules are quasi-projective. If \(S\)
		is quasi-injective and \(0\neq U\leq S\), an isomorphism \(U\to S\)
		extends to an endomorphism of \(S\). It is surjective, hence injective,
		so its restriction being onto forces \(U=S\).
	\end{proof}
	
	The last observation agrees with the corresponding result for
	quasi-injective virtually semisimple modules in \cite{BDV18b}.
	
	\begin{proposition}\label{prop:morita}
		An equivalence of module categories preserves isosimplicity,
		isoartinianity and quasi-projectivity. Consequently, if \(D\) is a
		principal right ideal domain and \(A=M_t(D)\), the standard row module
		\(D^{1\times t}\) is isosimple and all of its finite powers are
		isoartinian.
	\end{proposition}
	\begin{proof}
		An equivalence is exact, preserves finite sums, and induces an
		order-isomorphism between the subobject lattices of corresponding
		objects. It preserves isomorphisms and the lifting property defining
		quasi-projectivity. This proves the first assertion.
		The usual Morita equivalence between \(D\) and \(M_t(D)\) sends
		\(D_D\) to the row module. Every right ideal of \(D\) is principal
		and isomorphic to \(D\), and every submodule of \(D^n\) is free of
		rank at most \(n\). These are the standard free-submodule and Morita
		theorems; see \cite{AF92,Lam99}.
	\end{proof}
	
	\begin{example}\label{ex:ore}
		This gives examples from noncommutative division rings and their
		Ore polynomial rings. If \(K\) is a division ring, \(\sigma\) an
		automorphism and \(\delta\) a \(\sigma\)-derivation, the ring
		\(D=K[x;\sigma,\delta]\), with \(xa=\sigma(a)x+\delta(a)\), admits
		the usual degree division algorithms and is a principal ideal domain
		on both sides. The row modules over \(M_t(D)\) therefore satisfy
		Proposition~\ref{prop:morita}. The relevant division theory is treated
		in \cite{Cohn06,MR01}.
	\end{example}
	
	These examples can also be inflated along any surjection
	\(R\to M_t(D)\). The submodules and their homomorphisms do not change
	when the action is viewed through the quotient. Thus the original ring
	need not be prime, hereditary or Noetherian. Upper triangular
	matrix rings mapping onto such a diagonal block give further examples.
	
	\subsection{A non-quasi-projective example with all finite powers}
	
	The next construction explains why the power question should not be
	replaced by the assertion that the square forces quasi-projectivity.
	We first record the commutative tensor argument that will be used.
	
	\begin{lemma}\label{lem:reflexivetransfer}
		Let \(D\) be a commutative coherent UFD and let \(X\) be an
		isosimple right module with \(\End(X)=D\). Suppose \({}_DX\) is flat
		and every submodule of a finite power of \(X\) is generated by \(X\).
		For every \(N\leq X^n\), there is a finitely generated reflexive
		\(D\)-submodule \(P\leq D^n\) such that
		\[
		N=P\otimes_DX,\qquad \Hom(X,N)=P.
		\]
		Consequently \(X^n\) is isoartinian if and only if every finitely
		generated reflexive \(D\)-module of rank at most \(n\) is free.
	\end{lemma}
	\begin{proof}
		The module \(N\) is Noetherian, so a finite sum \(X^q\to N\) is
		onto. Let \(M\leq D^n\) be the image of the same matrix over \(D\).
		Flatness identifies \(N=M\otimes_DX\) inside \(X^n\).
		Coherence implies that \(M\), \(M^*\), and \(P=M^{**}\) are
		finitely presented. In the fraction vector space, one has
		\[
		M^{**}=\{v:\lambda(v)\in D\text{ for all }\lambda\in M^*\}.
		\tag{*}
		\label{eq:dualdescription}
		\]
		The coordinate functionals show \(P\subseteq D^n\).
		Formula \eqref{eq:dualdescription} also shows that \(P\) is reflexive:
		every functional on \(M\) extends to \(P\), and restriction identifies
		\(P^*=M^*\).
		
		For an irreducible \(p\), the ring \(D_{(p)}\) is a DVR and
		\(M_{(p)}=P_{(p)}\). Duals commute with this localization because the
		modules are finitely presented. Hence, for \(c\in P/M\), its
		annihilator contains a finite set \(a_1,\ldots,a_l\) of gcd one.
		To see finiteness, start with a nonzero annihilator and, for each of
		its finitely many prime factors, choose an annihilator not divisible
		by that factor.
		
		The submodule \(\sum a_iX\) is \(fX\) for some nonzero \(f\in D\).
		Thus \(f\) divides every \(a_i\), and is a unit. Consequently
		\(\sum a_iX=X\), and \(c\otimes X=0\).
		It follows that \(M\otimes_DX=P\otimes_DX\) inside \(X^n\).
		
		To recover \(P\), let \(v\in D^n=\Hom(X,X^n)\) have image in
		\(P\otimes_DX\). It lies in the fraction span of \(P\): any
		fraction-linear equation annihilating \(P\), after clearing
		denominators, also annihilates \(v\).
		For \(\lambda\in P^*\), extend it to a fraction-linear row \(b/a\),
		where \(a\neq0\) and \(b\in D^{1\times n}\).
		Since \(bP\subseteq aD\), one has
		\(b(P\otimes_DX)\subseteq aX\). Hence
		\((bv)X\subseteq aX\), and \(bv\in aD\).
		Formula \eqref{eq:dualdescription} and reflexivity give \(v\in P\).
		Thus \(\Hom(X,N)=P\).
		
		Every finitely generated reflexive module of rank \(r\) embeds in
		\(D^r\) after clearing denominators. Applying the preceding recovery
		formula proves both directions of the final assertion.
	\end{proof}
	
	The coherent module facts used here are standard; see
	\cite{Bourbaki,Stacks}. Over a two-dimensional regular local ring,
	finitely generated reflexive modules are free, by the depth criterion and the
	Auslander--Buchsbaum formula \cite{Matsumura,Stacks}.
	
	\begin{theorem}\label{thm:nagata}
		There is a left Noetherian and right Noetherian domain \(R\), and a
		cyclic isosimple left \(R\)-module \(S\), such that every finite power
		of \(S\) is isoartinian but \(S\) is not quasi-projective.
		Its endomorphism ring is \(k[x,y]_{(x,y)}\), where \(k\) can be any
		field of characteristic zero.
	\end{theorem}
	\begin{proof}
		Put \(D=k[x,y]_{(x,y)}\), let \(t\) be another indeterminate, and set
		\[
		\Sigma=\{f\in D[t]\setminus0:\gcd(\text{coefficients of }f)=1\},
		\qquad B=D[t]_\Sigma.
		\]
		This is the gcd-primitive localization, not the localization requiring
		the coefficient ideal to be the unit ideal.
		Gauss's lemma makes \(\Sigma\) multiplicative. Every nonzero
		polynomial is a constant \(c\in D\) times a member of \(\Sigma\);
		thus every nonzero element of \(B\) is associated to a constant.
		Also \((a,b)B=\gcd(a,b)B\), since after division by the gcd the
		polynomial \(a+bt\) belongs to \(\Sigma\). Every nonzero ideal is
		principal: start with one nonzero constant in it, take minima of the
		valuations of its finitely many prime factors among all elements of
		the ideal, and realize these minima using a finite set. The Bezout
		calculation gives a generator for the whole ideal.
		Thus \(B\) is a PID and its ideals are \(cB\), \(c\in D\).
		Moreover \(B\cap\operatorname{Frac}(D)=D\).
		
		Let \(\delta=d/dt\), let \(R=B[z;\delta]\), and let \(S=B\), with
		the usual left \(B\)-action and \(z\) acting as \(\delta\).
		All \(B\)-ideals \(cB\) are stable under \(\delta\).
		Multiplication by \(c\) is an \(R\)-isomorphism \(S\to cB\), so
		\(S\) is isosimple. An \(R\)-endomorphism is multiplication by
		an element annihilated by \(\delta\). In characteristic zero,
		the constants of \(\operatorname{Frac}(D)(t)\) are
		\(\operatorname{Frac}(D)\), whence \(\End_R(S)=D\).
		The ring \(R\) is a two-sided Noetherian domain by the Ore-extension
		Hilbert basis theorem \cite{MR01}.
		
		The module \(B\) is flat over \(D\). If \(N\leq S^n\), choose
		finite \(B\)-generators of \(N\). Multiply each by a common
		denominator in \(\Sigma\) to obtain polynomial vectors in \(N\).
		For a polynomial vector of degree \(r\), its \(r\)-th derivative
		divided by \(r!\) is its leading coefficient and lies in \(N\).
		Subtract the corresponding leading term and repeat. Thus all its
		coefficient vectors lie in \(N\cap D^n\), and
		\(N=(N\cap D^n)B\). Each constant vector defines an \(R\)-map
		\(S\to N\). Lemma~\ref{lem:reflexivetransfer}, with left and right
		reversed, now applies. Since \(D\) is regular local of dimension two,
		all finite powers are isoartinian.
		
		It remains to check non-quasi-projectivity. Put
		\[
		u=\frac{t}{x+yt}\in B.
		\]
		The denominator belongs to \(\Sigma\). One has
		\(\delta(u)=x/(x+yt)^2\), so multiplication by \(\overline u\)
		defines an \(R\)-map \(S\to S/xS\). If it lifted to an endomorphism
		of \(S\), there would be \(c\in D\) with \(c-u\in xB\).
		But \(yu\equiv1\pmod{xB}\), giving
		\[
		yc-1\in xB\cap D=xD.
		\]
		Reducing modulo \((x,y)\) is a contradiction.
		The contraction equality follows directly from gcd-primitivity.
		This proves the theorem. Passing to \(R^{\op}\) gives a right-module
		example.
	\end{proof}
	
	\section{From the square to all finite powers}\label{sec:powers}
	
	In this section, we consider considers the following question: Let  \(S\) be isosimple such that
		$S^2$ is isoartinian. When  
		$S^n$ is isoartinian for every positive integer $n?$

	\begin{theorem}\label{thm:power-main}
	Let \(S\) be an isosimple $R$-module, and set \(E=\End_R(S)\).	Suppose that either
		\begin{enumerate}[label=\textup{(\roman*)}]
			\item \(E\) is commutative and Noetherian; or
			\item \(E\) is right semihereditary.
		\end{enumerate}
		Then every submodule of \(S^n\) is isomorphic to \(S^r\) for some
		\(0\leq r\leq n\), for every positive integer \(n\). In particular,
		all finite powers of \(S\) are isoartinian.
		
		In case \textup{(i)}, \(E\) is a regular UFD of Krull dimension at most
		two, and all its finitely generated projective modules are free.
		In case \textup{(ii)}, \(E\) is a principal right ideal domain and
		\(S\) is quasi-projective.
	\end{theorem}
	
	The proof is given in the next four subsections. The first gives the
	generation and tensor properties forced by the square hypothesis.
	The commutative proof then uses a three-term relation module to bound
	local embedding dimensions. The right semihereditary proof uses
	splitting of finitely generated right ideals. In particular, the
	second assertion requires no preliminary Noetherian assumption on \(E\).
	
	\subsection{Generation and tensor products}
	
	We say that a module is generated by \(S\) if it is a sum of images of
	homomorphisms from \(S\). The generator theory behind this terminology
	is developed in \cite{Wis91,Wis96}. The assertions we need follow
	directly from the square hypothesis.
	
	\begin{proposition}\label{prop:flat}
		Every submodule of a finite power of \(S\) is generated by finitely
		many images of \(S\), and the left \(E\)-module \({}_ES\) is flat.
		For every \(N\leq S^n\), there is a finitely generated right
		\(E\)-submodule \(M\leq E^n\) for which tensoring the inclusion
		identifies
		\[
		N=M\otimes_ES\subseteq E^n\otimes_ES=S^n.
		\]
	\end{proposition}
	\begin{proof}
		We first prove generation by induction on \(n\). The assertion for
		\(n=1\) is the definition of isosimplicity. For
		\(N\leq S\oplus S^{n-1}\), put
		\[
		U=N\cap(S\oplus0),\qquad V=\pi(N),
		\]
		where \(\pi\) projects onto the last \(n-1\) coordinates. If \(V=0\),
		then \(N=U\) is zero or isomorphic to \(S\). Otherwise, choose
		finitely many maps \(\alpha_j:S\to V\) whose images generate \(V\),
		using the induction hypothesis. Define
		\[
		P_j=\{(x,t)\in S^2:(x,\alpha_j(t))\in N\}.
		\]
		By Proposition~\ref{prop:criterion}, each \(P_j\) is zero or
		isomorphic to \(S\) or \(S^2\). The maps
		\[
		P_j\longrightarrow N,\qquad (x,t)\longmapsto(x,\alpha_j(t))
		\]
		have images containing \(U\), and their projections generate \(V\).
		Their images therefore generate \(N\). This proves the induction.
		
		To verify flatness, consider a finite relation
		\[
		\sum_{i=1}^m a_is_i=0,\qquad a_i\in E,\quad s_i\in S.
		\]
		The vector \((s_1,\ldots,s_m)\) belongs to the kernel of the
		\(R\)-map \(S^m\to S\) given by the row \((a_1,\ldots,a_m)\).
		That kernel is generated by finitely many maps \(S\to S^m\).
		Writing their columns as \((b_{ij})_i\), we obtain elements
		\(t_j\in S\) such that
		\[
		s_i=\sum_j b_{ij}t_j,\qquad
		\sum_i a_i b_{ij}=0\quad\text{for every }j.
		\]
		The finite-relation criterion for flatness now proves that
		\({}_ES\) is flat; see \cite{Bourbaki,Rotman09}.
		
		Finally, choose finitely many maps \(S\to N\subseteq S^n\) whose
		images generate \(N\), and let \(M\leq E^n\) be generated by their
		columns. Flatness makes \(M\otimes_ES\to S^n\) injective. Its image
		is precisely \(N\), as required.
	\end{proof}
	
	We shall use the natural map
	\begin{equation}\label{eq:projective-recovery}
		\eta_P:P\longrightarrow\Hom_R(S,P\otimes_ES),\qquad
		\eta_P(p)(s)=p\otimes s.
	\end{equation}
	It is an isomorphism whenever \(P\) is finitely generated projective:
	this is immediate for a finite free module, and the assertion passes
	to idempotent summands. This observation does not require
	\(S\) to be projective over \(R\).
	
	\begin{lemma}\label{lem:smallprojectives}
		Every finitely generated projective right \(E\)-module of rank at
		most two is free.
	\end{lemma}
	\begin{proof}
		Let \(Q\) be the right quotient division ring of the right Ore domain
		\(E\), and put \(r=\dim_Q(P\otimes_EQ)\leq2\).
		View \(P\) as a direct summand of \(E^d\). In the corresponding
		\(r\)-dimensional subspace of \(Q^d\), choose \(r\) coordinate
		projections whose combined restriction is injective. Restricting to
		\(P\) gives a monomorphism \(P\to E^r\); injectivity follows because
		\(P\to P\otimes_EQ\) is injective. If \(r=0\), this already gives
		\(P=0\).
		
		Flatness from Proposition~\ref{prop:flat} gives
		\(P\otimes_ES\hookrightarrow S^r\). Since \(r\leq2\),
		Proposition~\ref{prop:criterion} yields
		\(P\otimes_ES\cong S^t\) for some \(t\leq r\).
		Applying \eqref{eq:projective-recovery}, we obtain
		\[
		P\cong\Hom_R(S,S^t)\cong E^t.
		\]
		Comparison of ranks over \(Q\) gives \(t=r\).
	\end{proof}
	
	\subsection{Three generators and local embedding dimension}
	
	The square hypothesis also determines the kernel of a row with three
	entries over \(E\). The commutative proof depends on this low-rank
	consequence, rather than on a bound for the stable rank of \(E\).
	
	\begin{lemma}\label{lem:three-row}
		For \(a,b,c\in E\), not all zero,
		\[
		\ker\bigl((a,b,c):E^3\longrightarrow E\bigr)\cong E^2.
		\]
	\end{lemma}
	\begin{proof}
		Permute coordinates so that \(a\neq0\), and let
		\[
		K=\ker\bigl((a,b,c):S^3\longrightarrow S\bigr).
		\]
		Projection onto the last two coordinates is injective on \(K\),
		because \(a\) is injective. Right Ore divisibility gives
		\[
		au=bv,\qquad au'=cv',\qquad v,v'\neq0.
		\]
		If \(b=0\), take \(u=0,v=1\), and use the analogous convention
		if \(c=0\). The maps
		\[
		s\longmapsto(-u(s),v(s),0),\qquad
		t\longmapsto(-u'(t),0,v'(t))
		\]
		embed two independent copies of \(S\) into \(K\). Its image in
		\(S^2\) therefore has uniform dimension two. By
		Proposition~\ref{prop:criterion}, \(K\cong S^2\).
		Applying the left exact functor \(\Hom_R(S,-)\) identifies the row
		kernel over \(E\) with \(\Hom_R(S,K)\cong E^2\).
	\end{proof}
	
	\begin{proposition}\label{prop:commregular}
		Suppose that \(E\) is commutative and Noetherian. Then \(E\) is a
		regular UFD of Krull dimension at most two, and
		\(\gl E\leq2\).
	\end{proposition}
	\begin{proof}
		Corollary~\ref{cor:UFD} shows that \(E\) is a UFD, so it is normal.
		Fix a maximal ideal \(\mathfrak m\), and put
		\[
		A=E_{\mathfrak m},\qquad \mathfrak n=\mathfrak mA,
		\qquad k=A/\mathfrak n.
		\]
		We claim that
		\begin{equation}\label{eq:embedding-bound}
			\dim_k\mathfrak n/\mathfrak n^2\leq2.
		\end{equation}
		Suppose instead that \(a,b,c\in\mathfrak n\) have linearly
		independent images in \(\mathfrak n/\mathfrak n^2\). By multiplying
		all three by a common denominator outside \(\mathfrak m\), we may
		take them to come from elements of \(E\). Multiplication by this
		unit preserves the independence of their images. Localization of
		Lemma~\ref{lem:three-row} gives
		\[
		L=\ker\bigl((a,b,c):A^3\to A\bigr)\cong A^2.
		\]
		For every \((x,y,z)\in L\), reduction of
		\(ax+by+cz=0\) modulo \(\mathfrak n^2\) shows that
		\(x,y,z\in\mathfrak n\). Hence
		\begin{equation}\label{eq:relations-in-m}
			L\subseteq\mathfrak nA^3,
			\qquad \mathfrak nL\subseteq\mathfrak n^2A^3.
		\end{equation}
		
		The three Koszul relations
		\[
		u_1=(-b,a,0),\qquad u_2=(-c,0,a),\qquad u_3=(0,-c,b)
		\]
		belong to \(L\). Their images in \(L/\mathfrak nL\) are linearly
		independent. Indeed, if \(\sum_i\lambda_i u_i\in\mathfrak nL\),
		the first coordinate and \eqref{eq:relations-in-m} give
		\[
		-\overline{\lambda}_1\overline b
		-\overline{\lambda}_2\overline c=0
		\quad\text{in }\mathfrak n/\mathfrak n^2.
		\]
		Thus \(\lambda_1,\lambda_2\in\mathfrak n\). The second coordinate
		then gives \(\lambda_3\in\mathfrak n\), using the nonzero image of
		\(c\). It follows that \(\dim_kL/\mathfrak nL\geq3\), contrary to
		\(L\cong A^2\). This proves \eqref{eq:embedding-bound}.
		
		The Noetherian local domain \(A\) is normal and has dimension at most
		its embedding dimension, hence at most two. If its dimension is zero,
		it is a field. If its dimension is one, it is a DVR by the normal
		local criterion \cite[Tag 00PD]{Stacks}. If its dimension is two,
		\eqref{eq:embedding-bound} forces its embedding dimension to equal
		its dimension, so it is regular local. Thus every \(E_{\mathfrak m}\)
		is regular local of dimension at most two. The regular local global
		dimension theorem and its local--global form give
		\(\gl E\leq2\); see \cite[Tag 065U]{Stacks} or
		\cite{Matsumura}. The asserted dimension bound and regularity follow
		as well.
	\end{proof}
	
	\begin{proposition}\label{prop:commprojectives}
		If \(E\) is commutative and Noetherian, every finitely generated
		projective \(E\)-module is free.
	\end{proposition}
	\begin{proof}
		Since \(E\) is a domain, a finitely generated projective module has
		constant rank. Ranks zero, one and two are covered by
		Lemma~\ref{lem:smallprojectives}. If the rank is at least three,
		Proposition~\ref{prop:commregular} shows that it is greater than
		\(\dim E\). Serre's splitting theorem then gives a free direct
		summand of rank one; see \cite{Bass68} or
		\cite[Tag 0GV7, Theorem 15.129.5]{Stacks}.
		Continue splitting until the remaining projective module has rank
		at most two, and apply Lemma~\ref{lem:smallprojectives} again.
	\end{proof}
	
	\subsection{The commutative Noetherian case}
	
	It remains to pass from projective coefficient modules to arbitrary
	submodules of \(S^n\). The following elementary observation accounts
	for the part lost upon taking a double dual.
	
	\begin{lemma}\label{lem:codimension-two}
		Assume that \(E\) is commutative and Noetherian. If a finitely
		generated \(E\)-module \(T\) satisfies
		\[
		T_{\mathfrak p}=0\quad\text{for every prime }\mathfrak p
		\text{ of height at most one},
		\]
		then \(T\otimes_ES=0\).
	\end{lemma}
	\begin{proof}
		There is nothing to prove for \(T=0\). Otherwise let
		\(I=\ann_E(T)\). Vanishing at the zero prime implies that \(T\)
		is torsion. Its finite generation then gives \(I\neq0\).
		For every height-one prime \(\mathfrak p\), finite generation and
		\(T_{\mathfrak p}=0\) give an element of \(I\) outside
		\(\mathfrak p\). Thus \(I\) is contained in no height-one prime.
		
		Choose generators \(a_1,\ldots,a_l\) of \(I\). Their gcd in the
		UFD \(E\) is a unit. Otherwise an irreducible common factor \(p\)
		would put \(I\) in the height-one prime \(pE\). Since every nonzero
		element of \(E\) acts injectively on \(S\), the submodule \(IS\)
		is nonzero. Write \(IS=fS\), where \(0\neq f\in E\), using
		isosimplicity. Each containment \(a_iS\subseteq fS\) gives
		\(a_i=fb_i\) with \(b_i\in E\). Hence \(f\) divides all the
		\(a_i\), and is a unit. Consequently \(IS=S\).
		
		For \(t\in T\) and \(s\in S\), write
		\(s=\sum_i a_is_i\). Since \(IT=0\) and \(E\) is commutative,
		\[
		t\otimes s=\sum_i ta_i\otimes s_i=0.
		\]
		This proves the assertion.
	\end{proof}
	
	\begin{theorem}\label{thm:comm-noetherian-powers}
		If \(E\) is commutative and Noetherian, every submodule of \(S^n\)
		is isomorphic to \(S^r\) for some \(0\leq r\leq n\). Hence every
		finite power of \(S\) is isoartinian.
	\end{theorem}
	\begin{proof}
		By Proposition~\ref{prop:flat}, write \(N\leq S^n\) as
		\(N=M\otimes_ES\), with \(M\leq E^n\) finitely generated. Put
		\(M^*=\Hom_E(M,E)\). The evaluation map \(M\to M^{**}\) is
		injective: the restrictions of the coordinate functionals of \(E^n\)
		already separate the elements of \(M\).
		
		We check projectivity of \(M^{**}\) explicitly. Choose a finite
		presentation \(E^a\to E^b\to M\to0\). Dualization gives
		\[
		0\longrightarrow M^*\longrightarrow E^b\longrightarrow E^a.
		\]
		If \(C\) is the cokernel of the last map, \(M^*\) is a second
		syzygy of \(C\). Since \(\gl E\leq2\), dimension shifting gives
		\(\Ext_E^1(M^*,X)=\Ext_E^3(C,X)=0\) for every \(X\).
		Thus \(M^*\) is projective. It is finitely generated by
		Noetherianity, so its dual \(M^{**}\) is finitely generated projective.
		By Proposition~\ref{prop:commprojectives},
		\[
		M^{**}\cong E^r.
		\]
		These dimension-shifting and duality facts are also treated in
		\cite{Rotman09,Stacks}.
		
		Let \(T=M^{**}/M\). At the zero prime the evaluation map is an
		isomorphism of finite-dimensional vector spaces. At a height-one
		prime \(\mathfrak p\), the ring \(E_{\mathfrak p}\) is a DVR,
		and the finitely generated torsion-free module \(M_{\mathfrak p}\)
		is free. Duals commute with localization, since the modules concerned
		are finitely presented. Hence the localized evaluation map is an
		isomorphism there too. Thus \(T\) satisfies the hypothesis of
		Lemma~\ref{lem:codimension-two}, and \(T\otimes_ES=0\).
		
		Tensoring
		\[
		0\longrightarrow M\longrightarrow M^{**}\longrightarrow T
		\longrightarrow0
		\]
		and using flatness yields
		\[
		N\cong M\otimes_ES\cong M^{**}\otimes_ES\cong S^r.
		\]
		Since \(N\leq S^n\), uniform dimensions give \(r\leq n\).
		Proposition~\ref{prop:criterion} proves isoartinianity.
	\end{proof}
	
	\begin{corollary}\label{cor:comm-obstructions}
		Let \(S\) be any isosimple module whose endomorphism ring \(E\)
		is commutative and Noetherian. If \(E\) is not regular, has Krull
		dimension at least three, or has a nonfree finitely generated projective
		module, then \(S^2\) is not isoartinian.
	\end{corollary}
	\begin{proof}
		Each alternative contradicts Proposition~\ref{prop:commregular} or
		Proposition~\ref{prop:commprojectives} under the square hypothesis.
	\end{proof}
	
	Theorem~\ref{thm:comm-noetherian-powers} does not require \(R\) to be
	commutative, nor does it require \({}_ES\) to be finitely generated
	or to have a nonzero dual. This distinction matters in
	Theorem~\ref{thm:nagata}. That example also shows that one cannot add
	quasi-projectivity to the conclusion in the commutative Noetherian case.
	
	\subsection{The right semihereditary case}
	
	Recall that \(E\) is right semihereditary if every finitely generated
	right ideal is projective. The equivalent finite-submodule formulation
	is discussed in \cite{Lam99,Wis91}. We retain the square hypothesis
	throughout.
	
	\begin{theorem}\label{thm:semihereditary-powers}
		If \(E\) is right semihereditary, it is a principal right ideal domain.
		Every submodule of \(S^n\) is isomorphic to \(S^r\), \(0\leq r\leq n\).
		Moreover \(S\) is quasi-projective.
	\end{theorem}
	\begin{proof}
		Let \(I\) be a nonzero finitely generated right ideal. It is
		projective and has rank one over the right quotient division ring.
		Lemma~\ref{lem:smallprojectives} gives \(I\cong E\), so \(I\) is
		principal. Hence \(E\) is right Bezout.
		
		Lemma~\ref{lem:basic} and the Noetherian property of \(S\) show that
		\(E\) satisfies the ascending chain condition on principal right
		ideals. A right Bezout ring with this condition is right Noetherian:
		if a right ideal \(J\) were not finitely generated, choose successively
		\(a_{j+1}\in J\setminus\sum_{i\leq j}a_iE\). The finitely
		generated right ideals \(\sum_{i\leq j}a_iE\) are principal and
		form a strictly ascending chain. Thus every right ideal is finitely
		generated, and therefore principal. This proves that \(E\) is a
		principal right ideal domain. The same Bezout-to-principal argument
		for endomorphism rings of isosimple modules appears in
		\cite[Corollary 2.5]{FN19}.
		
		Every finitely generated submodule \(M\leq E^n\) is free of rank
		at most \(n\). For clarity, project \(M\) onto the last coordinate,
		with image \(I\). If \(I=0\), use induction in \(E^{n-1}\).
		Otherwise \(I\cong E\), so the sequence
		\[
		0\longrightarrow M\cap E^{n-1}\longrightarrow M
		\longrightarrow I\longrightarrow0
		\]
		splits. The kernel is finitely generated, as a direct summand of the
		finitely generated module \(M\). Induction now proves the assertion.
		Proposition~\ref{prop:flat} consequently gives
		\(N\cong M\otimes_ES\cong S^r\) for each \(N\leq S^n\).
		
		We finish by proving quasi-projectivity without identifying it with
		the power conclusion. Every epimorphism
		\(q=(f,g):S^2\to S\) splits. Indeed, write
		\(fE+gE=hE\). Then
		\[
		S=fS+gS=hS,
		\]
		so \(h\) is an invertible endomorphism, and \(fE+gE=E\).
		Choose \(a,b\in E\) with \(fa+gb=1\); the column \((a,b)\)
		is a section of \(q\).
		
		Now let \(U\leq S\) and let \(v:S\to S/U\) be any homomorphism.
		The cases \(U=0\) and \(U=S\) are immediate. Otherwise form
		\[
		P_v=\{(x,y)\in S^2:v(x)=y+U\}.
		\]
		Projection onto the first coordinate gives an epimorphism
		\(P_v\to S\) with nonzero kernel \(0\oplus U\).
		By Proposition~\ref{prop:criterion}, \(P_v\) is isomorphic to
		\(S\) or \(S^2\). The first possibility contradicts the Hopf
		property of \(S\), since the projection is surjective with nonzero
		kernel. Thus \(P_v\cong S^2\), and its projection splits by the
		preceding paragraph. The second coordinate of a section gives a lift
		\(S\to S\) of \(v\). Hence \(S\) is quasi-projective.
	\end{proof}
	
	\begin{proof}[Proof of Theorem~\ref{thm:power-main}]
		Case \textup{(i)} is Theorem~\ref{thm:comm-noetherian-powers}, together
		with Propositions~\ref{prop:commregular} and
		\ref{prop:commprojectives}. Case \textup{(ii)} is
		Theorem~\ref{thm:semihereditary-powers}. In either case, a descending
		chain in \(S^n\) has terms of the form \(S^{r_i}\), with
		\(n\geq r_i\geq r_{i+1}\geq0\). These integers eventually
		stabilize, so the isomorphism types do as well.
	\end{proof}
	
	Finally, we leave the following question.
	\begin{question}\label{qu:final}
		For an arbitrary isosimple module \(S\), does the isoartinianity of
		\(S^2\) imply the isoartinianity of every finite power of \(S\)?
	\end{question}
	
	Theorem~\ref{thm:power-main} answers this question when the
	endomorphism ring is commutative Noetherian or right semihereditary.

\end{document}